\documentclass[11pt,reqno]{amsart}
\usepackage{amsmath,amsthm,amsfonts,amssymb,mathrsfs,stmaryrd,color}
\usepackage[all]{xy}
\usepackage{url}
\usepackage{geometry}
\usepackage{amsmath}
\usepackage{amsthm}
\usepackage{amsfonts}
\usepackage{amssymb}
\usepackage{tikz-cd}
\usepackage[numbers]{natbib}
\usepackage{mathrsfs}
\usepackage{enumerate}
\usepackage[utf8]{inputenc}
\usepackage[T1]{fontenc}
\usepackage{hyperref}
\usepackage[capitalize]{cleveref}
\usepackage{enumitem}
\setlist[enumerate]{itemsep=2pt,parsep=2pt,before={\parskip=2pt}}
\let\amsamp=&
\makeatletter
\newcommand{\colim@}[2]{%
  \vtop{\m@th\ialign{##\cr
    \hfil$#1\operator@font colim$\hfil\cr
    \noalign{\nointerlineskip\kern1.5\ex@}#2\cr
    \noalign{\nointerlineskip\kern-\ex@}\cr}}%
}
\newtheorem{theorem}{Theorem}[section]
\newtheorem*{theorem*}{Theorem}
\newtheorem*{definition*}{Definition}
\newtheorem{proposition}[theorem]{Proposition}
\newtheorem{lemma}[theorem]{Lemma}
\newtheorem{corollary}[theorem]{Corollary}

\theoremstyle{definition}
\newtheorem{definition}[theorem]{Definition}

\newtheorem{remark}[theorem]{Remark}

\newtheorem{example}[theorem]{Example}

\newtheorem{construction}[theorem]{Construction}
\newtheorem{claim}[theorem]{Claim}

\renewcommand{\lim}{\varprojlim}
\newcommand{\coker}[1]{\text{coker}(#1)}
\renewcommand{\to}{\rightarrow}

\renewcommand{\over}[2]{\stackrel{#1}{#2}}
\newcommand{\cal}[1]{\mathcal{#1}}

\newcommand{\otimesl}{\otimes^{\textbf{L}}}

\newcommand{\bb}[1]{\mathbb{#1}}
\newcommand{\colim}[1]{\underset{#1}{\text{colim}} \;}
\begin{document}
\title{Faithfully flat maps ring maps are not descendable}
\author{Ivan Zelich}
\begin{abstract}We construct a faithfully flat algebra over the infinite polynomial ring on an algebraically closed field that is not descendable.
\end{abstract}
\maketitle
\section{Introduction}
The notion of descendability was introduced by Akhil Matthew \cite{DescendMatthew} building upon the work of Paul Balmer \cite{DescendBalmer}. For the purposes of this exposition, we shall think of a descendable map of rings $f: A \to B$ as a strong version of descent. To be more precise, $f$ is descendable if the map $\{A\}_n \to \{ \text{Tot}_n B^{\bullet}\}_n$ is a pro-isomorphism of $A$-modules where $B^{\bullet}$ is the derived cech nerve and on the other hand, for descent, one usually requires the weaker condition that $A \to \text{Rlim}B^{\bullet}$ is an isomorphism. For a more detailed study of descendability we refer the reader to \cite{DescendMatthew}.\\
\indent Descendability of a ring map has better stability properties than of descent; it behaves more closely to \textit{universal descent}, see \cite{BhattScholzeProj} for applications to algebraic geometry, and especially \cite[D.3.6.2]{LurieSGA} for an $\infty$-category viewpoint. In this paper, we address the problem of whether faithfully flat ring maps are descendable. In concrete terms, if $f: A \to B$ is a faithfully flat map of rings, we have an exact sequence $A \to B \to \coker{f}$ corresponding to an element $\eta \in \text{Ext}_A^{1}(\coker{f}, A)$. The map $f$ will be descendable precisely when there exists an integer $n$ such that $\eta^{\otimes_A n} =0$.\\
\indent When $A$ is a Noetherian ring of finite Krull dimension, then the question is actually simple; it follows from the fact that the projective dimension of flat modules is finite.\cite{GrusonFlat} Similarly, if $A$ is an $\aleph_n$-countable ring, then any faithfully flat ring map is descendable.\cite[D.3.5]{LurieSGA}.\\
\indent It is an open problem\footnote{See \cite[11.24]{BhattScholzeProj}, \cite[Prop. 3.32]{DescendMatthew}, \cite[D.3.3.4]{LurieSGA}} whether an arbitrary faithfully flat map of rings is descendable, we will show that this is \textbf{not} the case. In particular:
\begin{theorem}[See \ref{corollary:counter}]\label{intro1}
There exists a ring map $A \to B$ that is faithfully flat but not descendable. In fact, $A$ can be made to be equal to $k[x_1, x_2, ...]$ with $k$ an algebraically closed field and $|k|=\beth_{\infty+}$.
\end{theorem}
This result is rather satisfactory considering the above discussion of the already known, positive results; note however that $\beth_{\infty} \ge \aleph_{\infty}$, so our bounds increase in a different manner unless the generalised continuum hypothesis is assumed.\\
\indent In a subsequent paper, we will prove the existence of other counterexamples by a more through investigation of the \textit{indivisible sequences} introduced in Section~\ref{indiv}.
\subsection*{Acknowledgements}
I would like to thank my advisor Aise Johan De Jong for our many helpful discussions and especially for suggesting to look into infinite Ramsey theory. I would also like to thank Mehtaab Sawhney for insightful discussions regarding Ramsey theory.
\newpage
\section{The counterexample}
\subsection{Preliminaries}
We fix some notation that will be used throughout this article.
\begin{enumerate}
\item[(i)] If $S$ is a set and $R$ a ring, define $\text{Hom}^{\text{fin}}_{\text{Set}}(S, R) \subset \text{Hom}_{\text{Set}}(S,R)$ to be the set-theoretic maps with finite support i.e. $f \in \text{Hom}^{\text{fin}}_{\text{Set}}(S, R)$ if there exists a subset $K \subset S$ with $|K| < \infty$ such that $f(S\setminus K) = 0$. For each $t \in S$, set $\delta_{s=t} \in \text{Hom}^{\text{fin}}_{\text{Set}}(S, R)$ to be the function having $\delta_{s=t}(s)=1$ if $s=t$ and $0$ else.
\item[(ii)]Further, if $S_i$ are sets indexed by the set $\{1,2,...,n\}$, then we will identify the vector $1_{s_1,s_2,...,s_n}$ in the nested direct sum $\oplus_{S_1} \oplus_{S_2} \oplus ... \oplus_{S_n} R$ as follows: Every vector in $v \in \oplus_{S_1} \oplus_{S_2} \oplus ... \oplus_{S_n} R$ can be identified as a function $f_{v} \in \text{Hom}^{\text{fin}}_{\text{Set}}(\prod_i S_i, R)$, and under this identification, $1_{s_1,...,s_n}$ corresponds to $\delta_{t=s}$ where $s=(s_1,s_2,...,s_n) \in \prod_i S_i$.
\item[(iii)]For any set $S$ and $n \in \bb{N}$, we let $\text{pr}_i:S^{\times n} \to S$ be the projection onto the $i^{\text{th}}$-coordinate and $p_i: S^{\times n} \to S^{\times n-1}$ be defined as follows: For any $s = (s_1,s_2,..,s_n) \in S^{\times n}$,
\[p_i(s) = (s_1,s_2,...,s_{i-1},s_{i+1},...,s_n).\]
For a point $s = (s_1,...,s_{n-1}) \in S^{\times n-1}$, we consider the degeneracy map $t_i^{s}: S \to S^{\times n}$ which takes a point $s' \in S$ to the point:
\[(s_1, ..., s_{i-1}, s', s_{i},...,s_{n-1}) \in S^{\times n}.\]
\end{enumerate}
\indent We also recall the following form of Lagrange's interpolation theorem.
\begin{theorem}[Lagrange interpolation]
Let $k$ be a field and $f \in k[x]$ a polynomial of degree $n$. Then there exists a universal polynomial $P \in k[x_1,...,x_n]$ such that for $z \in k^{\times n+1}$ we have:
\[\sum_{i=1}^{n+1} (-1)^iP(p_i(z))f(\text{pr}_i(z)) = 0.\]
Explicitely, $P = \prod_{ 1 \le i < j \le n} (x_j - x_i)$.
\end{theorem}
\begin{definition}
Let $k$ be a field. For $z \in k^{\times n+1}$, we may define a $k$-linear operator 
\[\nabla^z_i: \text{Hom}(k^{\times m}, k) \to \text{Hom}(k^{\times m-1},k)\]
as follows: For any function $f \in  \text{Hom}(k^{\times m}, k)$, set:
\[\forall s \in k^{\times m-1}: \; \;  \nabla^{z}_i(f)(s) = \sum_{j=1}^{n+1} (-1)^j P(p_j(z)) f(t^{s}_i(\text{pr}_j(z))).\]
For any two $z \in k^{\times n+1}, z' \in k^{\times n'+1}$, one can verify that:
\[\nabla^z_i \nabla^{z'}_j = \nabla^{z'}_j \nabla^{z}_i\]
for $i \neq j$. Furthermore, if $f(x_1,...,x_m) = \prod_i f_i(x_i)$, and $z^i \in \text{Hom}(k^{\times n_i}, k)$ for $i\in \{1,2,...,m\}$, 
\[(\nabla^{z^m}_m \circ \nabla^{z^{m-1}}_{m-1} \circ ... \circ \nabla^{z^1}_{1})(f) = \prod_i \nabla^{z^i} f_i.\]
\end{definition}
\begin{corollary}\label{corollary:interpol}
Let $f: \Omega \to k$ be function defined on a countably infinite subset $\Omega \subset k$. Then: $\nabla^{z} f = 0,\;  \forall z \in k^{\times n+1}$ if and only if $f$ is a polynomial of degree $\le n$.
\end{corollary}
Finally, we show:
\begin{proposition}\label{proposition:notpoly}
For any field $k$, there exists an algebraically closed field extension $k \subset E(k)$, and a function $f: E(k) \to E(k)$ such that there is no countably infinite subset $S \subset E(k)$ such that $f$, when restricted to $S$, is a polynomial. If $|k| = \infty,$ then $|E(k)| = |k|$.
\end{proposition}
\begin{proof}
We define a sequence of field extensions $0=k_0\subset k=k_1 \subset k_2 \subset ...$ and functions $f_{n}:k_n \to k_{n+1}$ inductively. For the fields, set 
\[k_{n+1} = \overline{k_{n}(k_{n} \setminus k_{n-1})}\]
That is, $k_{n+1}$ is the algebraic closure of the field of fractions of the free polynomial ring $k_{n}[X_s]$ indexed by elements $s \in k_{n} \setminus k_{n-1}$. For the functions, we take:
\[\forall s \in k_{n} \setminus k_{n-1}: \; \; f_n(s) = X_s \in k_{n+1}\]
\[\forall s \in k_{n-1}: \; \; f_n(s) = f_{n-1}(s) \in k_{n} \subset k_{n+1}.\]
Set $E(k) = \underset{n}{\cup} k_n$, and $f=\colim{n} f_n$, which is a map $E(k) \to E(k)$. \\
\indent We claim that $f$ is not polynomial on any infinite subset $S \subset E(k)$. Suppose that there exists a polynomial $p$ of degree $d$ such that $f|_{S} = p|_{S}$. There are two cases; there exists a $t \in \bb{N}$ such that $|S \cap k_t\setminus k_{t-1}| \ge d+1$, or for any $n \in \bb{N}$, $S \cap (k \setminus k_n) \neq 0$.\\
\indent In the first case, there exists an element $z \in (k_t \setminus k_{t-1})^{d+1}$ such that:
\[\nabla^z f = \sum_i (-1)^iP(p_i(z))f(\text{pr}_i(z))= 0.\]
As $P(p_i(z)) \in k(z_1,...,z_{n+1})$ for every $i$, this is a contradiction by construction.\\
\indent In the second case, take an $N \in \bb{N}$ such that the coefficients of $p$ are contained in $k_N$, and an $s \in S \cap (k \setminus k_N)$. Suppose that $s \in k_n\setminus k_{n-1}$. By supposition, $f(s) = p(s) \in k_n$, which is again a contradiction by construction.
\end{proof}
Later, we will need some results of infinite Ramsey theory, which originated from the paper \cite{ErdosRado}. Before stating the main theorem in a more convenient language, let us introduce some notation. 
\begin{construction}\label{notation}
Define the \textit{beth numbers} $\beth_n$ by the following formula:
\[\beth_0 = \aleph_0, \; \beth_n = 2^{\beth_{n-1}}.\]
\end{construction}
Note that in particular, $\beth_{n} \ge \aleph_{n}$, where equality happens when the generalised continuum hypothesis is assumed. For any set $S$, define $[S]^r=\{K \subset S: \; |K|=r\}$, and for any cardinal $\alpha$, define $\alpha_{+}$ to be its successor.
\begin{theorem}[Theorem 7.3 \cite{Higherinfinite}]\label{theorem:erdosrado}
For any set $S$ of size $\beth_{r+}$, and any set map:
\[f: [S]^{r+1} \to \bb{N}\]
There exists a set $S' \subset S$ with $|S|=\aleph_{1}$ and an integer $n \in \bb{N}$ such that $[S']^{r+1} \subset f^{-1}(n)$.
\end{theorem}
\begin{remark}\label{remark:boxuseful}
Theorem~\ref{theorem:erdosrado} also implies the following weaker statement: For a set $S$ of cardinality $\beth_{r+}$, and a set theoretic map
\[f: S^{\times r+1} \to \bb{N}\]
such that $f(s) = f(\sigma(s))$ for every $s \in S^{\times r+1}$ and $\sigma \in \mathrm{S}_{r+1}$, the symmetric group on $r+1$-elements, there exists a $n \in \bb{N}$ and countably infinite subsets $S_1, ..., S_{r+1}$ of $S$ such that $ \prod_{i=1}^{r+1} S_i \subset f^{-1}(n)$.
\end{remark}
\subsection{Cup-products and descendability}\label{indiv} Let $k$ be a commutative ring, $R$ be a $k$-algebra. For $n \in \bb{N} \cup \{\infty\}$, let $R_n$ be the ring $R^{\otimes_k n}$ and
\[t_i: R \to R_n\]
to be the inclusion by $1$ into the $i^{\text{th}}$ factor. It will be convenient to extend the definition of $t_i$ to $R\text{-mod}$ by setting $t_i(M)=M \otimes_{R, t_i} R_n$. \\
\indent Let $\{x_s\}_{s \in S}$ be a set of elements of $R$ indexed by a set $S$. Consider the map
\[f_{R,S}: \bigoplus_{s \in S} R \to R \oplus \bigoplus_{s \in S} R\]
determined by $f_{R,S}(1_{s})_{s'} = x_s\delta_{s'=s}$ and $f_{R,S}(1_s)_{0} = 1 \in R$ for all $s\in S$, where the factor of the RHS that is not indexed by $S$ we view as the $0^{\text{th}}$-component. We note that:
\begin{lemma}\label{lemma:annihilators}
The map $f_{R,S}$ is injective as soon as $\text{Ann}_{R}(x_{s}) \cap \text{Ann}_{R}(x_{s'})$ for any distinct $s, s' \in S$.
\end{lemma}
\begin{definition}
Let $k$ be a commutative ring, $R$ be a $k$-algebra and $S$ a set. A sequence of elements $\{x_s\}_{s \in S}$ of $R$ is said to be \textit{indivisible relative to k} if:
\begin{enumerate}
\item[(i)] $f_{R,S}$ is $R$-universally injective. Equivalently, $f_{R,S}$ is injective and $\coker{f_{R,S}}$ is $R$-flat.
\item[(ii)] For any natural number $n \ge 1$, and a finite collection $(s_1,s_2,...,s_n) \in S^{\times n}$, the ideal $\sum_{i=1}^n t_i(x_{s_i})$ in $R^{\otimes_{k} n}$ is not the unit ideal.
\end{enumerate}
When $k=\bb{Z}$, we simply say that $\{x_s\}_{s \in S}$ is \textit{indivisible}.
\end{definition}
\begin{proposition}\label{proposition:indivequiv}
Let $R$ be a commutative ring with a sequence of elements $\{x_s\}_{s \in S}$ such that:
\begin{enumerate}
\item[(i)] For any pair of $s, s' \in S$, the ideal $(x_s, x_s') \subset R$ is the unit ideal.
\item[(ii)] For every $s \in S$, the composed map $k \to R \to R/x_s$ is faithfully flat.
\end{enumerate}
Then $\{x_s\}_{s \in S}$ is an indivisible sequence in $R$ relative to $k$.
\end{proposition}
\begin{proof}
We first observe that property (i) assures that $f_{R,S}$ is injective due to Lemma~\ref{lemma:annihilators}. Since $f_{R,S} \otimes_{R} R' = f_{R',S}$ for any ring map $f: R \to R'$, it suffices to observe that $\{f(x_s)\}_{s \in S}$ still satisfy (i).\\
\indent Using the Kunneth spectral sequence applied to the complexes $R \over{x_s}{\to} R$, we can show that:
\[R^{\otimes_{k} n}/(t_1(x_{s_1}),...,t_n(x_{s_n})) = \bigotimes_{i=1, k}^n R/x_{s_i}.\]
Under the faithfully flat assumption of (ii), the latter ring is non-zero.\\
\end{proof}
\begin{example}\label{example:indiv}
\begin{enumerate}
\item[(i)] Let $k$ be a field, $R=k[x]$, and for a subset $S \subset k$, sequence $x_s=x-s$ is indivisible relative to $k$.
\item[(ii)] For a set $S$, let $R=\text{Hom}_{\text{Set}}(S, \bb{F}_2)$. The sequence $x_s = 1-\delta_{t=s}$ of idempotents of $R$ is indivisible.
\end{enumerate}
\end{example}
\begin{construction}\label{construction:classes}
Let $R$ be a $k$-algebra with an $S$-indexed sequence that is indivisible relative to $k$, and let $f_{R,S}$ be the associated map. Let $\eta \in \text{Ext}^1_{R}(\coker{f_{R,S}}, \oplus_{s \in S} R)$ be the extension class corresponding to the exact sequence defined by $f_{R,S}$.  For each $i \in 1, ..., n$, let $\eta_i = \eta \otimes_{R, t_i} R_n$.
\end{construction}
\begin{remark}
Throughout this section we will be tensoring various extension classes. We remark that for a ring $A$ together with $A$-modules $M,N,M',N'$, there is an $A$-bilinear pairing:
\[\text{RHom}_A(M,N) \times \text{RHom}_A(M',N') \to \text{RHom}_A(M \otimesl_A M', N \otimesl_A N').\]
If $N=N'=A$, then we will compose with the multiplication isomorphism $A \otimesl_A A \to A$ so that the target is $\text{RHom}_A(M \otimesl_A M', A)$. Moreover, as is the case in this section, suitable flatness conditions remove the necessity for derived tensor products on the right.
\end{remark}
The following theorem is the main technical input for our construction of faithfully flat ring maps that are not descendable.
\begin{theorem}\label{theorem:countermod}
For any $n \in \bb{N} \cup \{\infty\}$, let $F$ be a field with $|F|= \beth_{(n-1)+}$, and $k=E(F)$ (see \ref{proposition:notpoly}). Let $R=k[x]$ and $\{x-s\}_{s \in k}$ be the $k$-indexed indivisible sequence relative to $k$ (see \ref{example:indiv}). Let $\eta_i$ be the classes of Construction~\ref{construction:classes}. Then there exists an $R$-linear map $h: \oplus_{k} R \to R$ such that:
\[\bigotimes_{i=1, R_n}^m (h \otimes_{R, t_i} R_n) \circ \eta_i \neq 0,\]
for any $m < n+1$.
\end{theorem}
In proving this result, one must show certain $\text{Ext}^n$ classes are non-zero, which is in general difficult. The following lemma provides a simplified criterion.
\begin{lemma}\label{lemma:extcrit}
Let $A$ be a ring, $P_1, P_2, P_1', P_2', ..., P_r'$ be projective $A$-modules, together with two flat $A$-modules $M$ and $M'$ that form two exact sequences:
\[0 \to P_1 \stackrel{d_1}{\to} P_2 \to M \to 0,\]
\[0 \to P_1' \stackrel{d_2}{\to} ... \to P_r' \to M' \to 0.\]
We may view these as extensions $\eta_1 \in \text{Ext}^1_A(M, P_1)$ and $\eta_2 \in \text{Ext}^{r-1}_A(M',P_1')$, and as projective resolutions $P^{\bullet}_1, P^{\bullet}_2$. Then the projective resolution:
\[P^{\bullet}_1 \otimes_A P^{\bullet}_2 \to M \otimes_A M',\]
gives rise to an extension $\text{Ext}^{r}_A(M \otimes_A M', P_1 \otimes_A P_1')$ which is equal to $\eta_1 \otimes_A \eta_2$. Moreover, showing this extension is non-zero is equivalent to verifying that 
\[P_1 \otimes_A P_1' \stackrel{d_1 \otimes 1_{P'_1} \oplus 1_{P_1} \otimes d_2}{\to} P_2 \otimes_A P_1' \oplus P_1 \otimes_A P_2\]
doesn't admit an $A$-linear left inverse.
\end{lemma}
\begin{proof}
For the first claim, the resolutions $P_1^{\bullet}, P_2^{\bullet}$ are cofibrant replacements for $M$ and $M'$, and the maps $P_1^{\bullet}[-(r-1)] \to P_1$ and $P_2^{\bullet}[-1] \to P_1'$ correspond to chain maps that are identities on degree 0. The tensor product of this map is then just the identity on $P_1 \otimes_A P_1'$ in $P_1^{\bullet} \otimes_A P_2^{\bullet}[-r]$, which corresponds to the extension class of $M\otimes_A M'$ by $P_1 \otimes_A P_1'$ given by $P_1^{\bullet} \otimes_A P_2^{\bullet}$.\\
\indent The final statement is just a general result of homological algebra.
\end{proof}
\begin{proof}[Proof of Theorem~\ref{theorem:countermod}]
We first reduce to the situation that $m=n$. For $n \in \bb{N}$, there are ring maps $\mu_{n,m}: R_n \to R_m$ that take an element $r_1 \otimes r_2 \otimes... \otimes r_n \in R_n$ to $ r_1 \otimes r_2 \otimes... \otimes (r_mr_{m+1}...r_{n}) \in R_m$. By base-change along $\mu_{n,m}$, or $\colim{j\ge m}{\mu_{j,m}}: R_{\infty} \to R_m$, we may reduce to the case $n=m$.\\
\indent Note the equality:
\[\bigotimes_{i=1, R_n}^n ((h \otimes_{R, t_i} R_n) \circ \eta_i) = \bigotimes_{i=1, R_n}^n (h \otimes_{R, t_i} R_n) \circ \bigotimes_{i=1, R_n}^n \eta_i.\]
By Lemma~\ref{lemma:extcrit}, showing the right-hand side is non-zero is equivalent to showing that there is no $R_n$-linear extension $r$ fitting into the diagram:
\[\begin{tikzcd} R_n & \\
\bigoplus_{k^{\times n}} R_n \arrow[r, "\oplus_i t_i(f_{R,S})"] \arrow[u, "h^{\otimes_k n}"] & \bigoplus_{i} \bigoplus_{k^{\times n-1}} R_n \oplus \bigoplus_{k^{\times n}} R_n \arrow[ul, swap, "r", dashed]\end{tikzcd}\]
Let $H: k \to k$ be a function that is not polynomial on any countably infinite subset of $k$ (see \ref{proposition:notpoly}). Set:
\[h(1_{s}) = H(s) \in k[x],\]
for every $s \in k$. Further, for every $s' \in k^{\times n-1}$ and $s \in k^{\times n}$, let
\[r(1_{i,s'}) = f_i(s'), r(1_{i,s}) = g_i(s).\]
Under the identification $R_n = k[x_1,...,x_n]$, $f_i(s'), g_i(s)$ are polynomials subject to the equation:
\[\sum^n_{i=1} f_i(p_i(t))(x_1,...,x_n) + \sum^n_{i=1} (x_i - \text{pr}_i(t)) g_i(t)(x_1,...,x_n) = \prod^n_{i=1} h(1_{\text{pr}_i(t)})(x_i).\]
for every $t \in k^{\times n}$. In particular, we have:
\[\sum^n_{i=1} f_i(p_i(t))(\text{pr}_1(t),...,\text{pr}_n(t)) = \prod^n_{i=1} h(1_{\text{pr}_i(t)})(\text{pr}_i(t)). \; \; (*)\]
\begin{claim}
There exists a countably infinite subsets $S_1, ..., S_n$ of $k$, such that for every $\omega \in \Omega =\prod^n_i S_i$, the polynomials $f_i(p_i(\omega))$ have (total) degree $\le m$ for some $m \in \bb{N}$.
\end{claim}
\begin{proof}
Consider the function:
\[c: k^{\times n} \to \bb{N}\]
which takes a point $t \in k^{\times n}$ to $\underset{i\in\{1,2,...,n\}}{\text{max}}(\text{deg}(f_i(p_i(t)))$. By Remark~\ref{remark:boxuseful}, there exists a $m \in \bb{N}$ and countably infinite subsets $S_1, ..., S_n$ such that $ \prod_{i=1}^{n} S_i \subset f^{-1}(m)$. Set $\Omega = \prod_{i=1}^n S_i$ to get the claim.
\end{proof}
For each $i \in \{1,2,...,n\}$, let us take points $z^i \in S_i^{\times m+1}$. Note that for all $i \in \{1,2,...,n\}$, we have functions $\tilde{f}_i, \tilde{h} \in \text{Hom}_{\text{Set}}(k^{\times n},k)$ defined as follows:
\[\forall t \in k^{\times n}: \; \; \tilde{f}_i(t) = f_i(p_i(t))(\text{pr}_1(t),...,\text{pr}_n(t))\]
\[\forall t \in k^{\times n}: \; \; \tilde{h}(t) = \prod_{i=1}^n H(\text{pr}_i(t)).\]
We see that $(*)$ may be re-written as an equality
\[\sum^n_{i=1} \tilde{f}_i = \tilde{h},\]
of functions in $\text{Hom}_{\text{Set}}(k^{\times n}, k)$. For any $s \in \prod_{j \in \{1,2,...,n\} \setminus \{i\}} S_j$, we see
\[\tilde{f}_i \circ t^{s}_i: k \to k\]
is a polynomial of degree $\le m$. Therefore:
\[\nabla^{z^i}_i \tilde{f}_i = 0,\]
for all $i \in \{1,2,...,n\}$. Therefore:
\[ \prod_i^n \nabla^{z^i} H = (\nabla^{z^n}_n \circ \nabla^{z^{n-1}}_{n-1} \circ ... \circ \nabla^{z^1}_1)( \tilde{h}) = (\nabla^{z^n}_n \circ \nabla^{z^{n-1}}_{n-1} \circ ... \circ \nabla^{z^1}_1)(\sum_{i=1}^n \tilde{f}_i)=0.\]
This implies that there exists an $i \in \{1,2,...,n\}$ such that $\nabla^{z^i} H = 0$ for any such $z^i \in S_i^{\times m+1} \subset k^{m+1}$. But this would imply that $H$ is a polynomial of degree $m$ on the infinite subset $S_i \subset k$ (due to Corollary~\ref{corollary:interpol}), which is a contradiction.
\end{proof}
\begin{corollary}\label{corollary:algcup}
Let $n \in \bb{N} \cup \{\infty\}$, $F$ a field with $|F| = \beth_{(n-1)+}$, and $k=E(F)$. There exists a faithfully flat ring maps $k[x_1,...,x_n] \over{f_i}{\to} C_i$ for $i \in \{1,2,...,n\}$ such that if $\eta_i \in \text{Ext}^1_{k[x_1,...,x_n]}(\coker{f_i}, k[x_1,...,x_n])$ is the class of the extension $k[x_1,...,x_n] \over{f_i}{\to} C_i \to \coker{f_i}$, then:
\[\bigotimes_{i=1, k[x_1,...,x_n]}^m \eta_i \neq 0\]
for all $m < n+1$.
\end{corollary}
\begin{proof}[Proof of \ref{corollary:algcup}]
Let $A=k[x_1,...,x_n]$. The output of Theorem~\ref{theorem:countermod} are flat $A$ modules $I_i$ and extension classes:
\[\tilde{\eta}_i \in \text{Ext}^1_{A}(I_i, A),\]
for all $i \in \{1,2,...n\}$, such that:
\[\bigotimes_{i=1, A}^{j} \tilde{\eta}_i \neq 0\]
for all $j < n+1$. Using Proposition~\ref{proposition:algtomod}, we obtain faithfully flat $A$-algebra maps $A \over{f_i}{\to} C_i$, and extensions $\tilde{\eta}^a_i \in \text{Ext}_A^1(\coker{f_i}, A)$, and maps $s_{\eta_i}: I_i \to \coker{f_i}$ such that:
\[\tilde{\eta}_i = \tilde{\eta}^a_i \circ s_{\eta_i}.\]
Therefore:
\[\bigotimes_{i=1, A}^{j} \tilde{\eta}^a_i \circ \bigotimes_{i=1, A}^{j} s_{\eta_i}= \bigotimes_{i=1, A}^{j} (\tilde{\eta}^a_i \circ s_{\eta_i})=\bigotimes_{i=1, A}^{j} \tilde{\eta}_i  \neq 0,\]
for all $j < n+1$. By setting $\eta_i= \tilde{\eta}^a_i$, we're done.
\end{proof}
We used the following proposition in the above proof.
\begin{proposition}\label{proposition:algtomod}
Let $A$ be a commutative ring and $M$ a flat $A$-module, and let $\eta$ be an element in $\text{Ext}^1_A(M,A)$, represented by a class $A \over{\eta}{\to} N \to M$. There exists an extension $A \over{\eta^a}{\to} B \to C$, with $C$ a flat $A$-module equipped with an $A$-linear morphism $s_{\eta}: N \to C$ such that $\eta^a$ is map of $A$-algebras, and $\eta = \eta^a \circ s_{\eta}$.
\end{proposition}
\begin{proof}
Fix an extension $\cal{E}:=  0\to A \over{\eta}{\to} N \to M \to 0$ of $M$ by $A$. Define
\[B= \colim{n}\text{Sym}^n_A(N)\]
where the transition maps are induced by the map $N^{\otimes_A n} \hookrightarrow N^{\otimes_A n+1}$ given by $\eta_n=\text{id}_{N^{\otimes_A n}} \otimes_A \eta$; note that each $\eta_n$ is injective due to the flatness of $N$. We endow $B$ with the natural commutative algebra structure, note that we have a `twisted' unit map $A \over{\eta^{a}}{\hookrightarrow} B$ induced by $\eta$.\\
\indent Define $M' = \coker{A \to B}$, $s_{\eta}: M \to M'$ the induced $A$-linear map and $M'_n = \coker{A \to \text{Sym}^n_A(N)}$; note that $M'= \colim{n} M'_n$. We claim that $M'$ is flat. Note that we have an exact sequence:
\[0 \to \text{Sym}_A^{n}(N)\over{\eta_{n}}{\to} \text{Sym}_A^{n+1}(N) \to \text{Sym}_A^{n+1}(M) \to 0\]
An application of the snake lemma then gives a sequence:
\[0 \to M'_n \to M'_{n+1} \to \text{Sym}_A^{n+1}(M) \to 0\]
By Lazard's theorem, $\text{Sym}_A^{n+1}(M)$ is flat for each $n\ge 0$ as $M$ is flat. Since $M'_1 = M$, we obtain the claim by induction.\\
\indent Hence, $A \to B$ is a faithfully flat ring map (as the cokernel is a faithfully flat $A$-module). Moreover, by the five lemma we see $N \simeq (M \to M' \leftarrow B)$, so the morphism of extensions $\cal{E} \to \cal{E'}$ induces a pullback $\text{Ext}^1_A(M',A) \to \text{Ext}^1_A(M,A)$ taking $\eta^{a} \mapsto \eta$. More precisely, we see that $\eta^a \circ s_{\eta} = \eta$.
\end{proof}
To complete the construction of our counterexample, we note:
\begin{proposition}\label{proposition:cup}
Let $A \over{f_{s}}{\to} C_{s}$ be a sequence of faithfully flat ring maps for $s \in \bb{N}$, and $\eta_s$ be the class in $\text{Ext}^1_A(\coker{f_s}, A)$ corresponding to the exact sequence $A \over{f_s}{\to} C_{s} \to \coker{f_s}$. Let $f=\otimes_{s, A}f_s: A \to \otimes_s C_s$ and denote by $\eta \in \text{Ext}^1_A(\coker{f},A)$ the extension class of $A \to \otimes_s C_s \to \coker{f}$. Then:
\[\otimes_{s, A} \eta_{s} \neq 0 \implies \eta^{\otimes_{A}n} \neq 0\]
for all $0 < n < \infty$. Furthmore, $f$ is faithfully flat.
\end{proposition}
\begin{proof}The claim that $f$ is faithfully flat is standard.\\
\indent Let $C = \otimes_{s \in \bb{N}} C_s$, we have a commutative diagram:
\[\begin{tikzcd}
& C_{s}\arrow[d, "t_s"] \\
A \arrow[ru, "f_s"] \arrow[r, swap, "f"] & \otimes_{s \in S} C_s.
\end{tikzcd}\]
Let $p_s: \coker{f_s} \to \coker{f}$ be the induced map, fitting into a commutative diagram,:
\[\begin{tikzcd}
& C_{s}\arrow[d, "t_s"] \arrow[r]& \coker{f_s} \arrow[d, "p_s"]\\
A \arrow[ru, "f_s"] \arrow[r, swap, "f"] & \otimes_{s \in S} C_s \arrow[r] & \coker{f}
\end{tikzcd}\]
In particular, we have $\eta \circ p_s = \eta_s$, which then results in the equation:
\[\eta^{\otimes n} \circ ( \otimes_{s' \in S'} p_s) = \otimes_{s' \in S'} \eta_{s'},\]
for any subset $S' \subset \bb{N}$ with $|S'| = n$.\\
\end{proof}
\begin{corollary}\label{corollary:counter}
$F$ a field with $|F| = \beth_{\infty +}$, and $k=E(F)$. There exists a faithfully flat ring map $k[x_1,...,x_{\infty}] \to A_{\infty}$ that is not descendable.
\end{corollary}
\begin{proof}
By Corollary~\ref{corollary:algcup}, for $i \in \bb{N}$, there are faithfully flat ring maps $f_i = k[x_1,x_2,...] \to A_i$ such that if $\eta_i \in \text{Ext}^1_{k[x_1,x_2,...]}(\coker{f_i}, k[x_1,x_2,...])$ is the class of the extension $k[x_1,x_2,...] \over{f_i}{\to} A_i \to \coker{f_i}$, then:
\[\otimes_{i=1, k[x_1,x_2,...]}^n \eta_i \neq 0, \; \; \forall n \in \bb{N}.\]
Let $A_{\infty}=\otimes_{i \in \bb{N}} A_i$ and $f=\otimes_{i} f_i: k[x_1,x_2,...] \to A_{\infty}$. Proposition~\ref{proposition:cup} implies that $f$ is faithfully flat and has the property that if $\eta \in \text{Ext}^1_{k[x_1,x_2,...]}(\coker{f}, k[x_1,x_2,...])$ is the class of the extension $k[x_1,x_2,...] \over{f}{\to} A_{\infty} \to \coker{f}$, then:
\[\eta^{\otimes_{k[x_1,x_2,...]} n} \neq 0, \; \; \forall n \in \bb{N}.\]
This means that $f$ is not descendable, as desired.
\end{proof}
\bibliographystyle{amsalpha}
\bibliography{ffdescent}
\end{document}